\documentclass{amsart}
\usepackage{graphicx}
\usepackage{verbatim}
\newtheorem{thm}{Theorem}[section]
\newtheorem{cor}[thm]{Corollary}
\newtheorem{lem}[thm]{Lemma}
\newtheorem{prop}[thm]{Proposition}
\theoremstyle{definition}
\newtheorem{defn}[thm]{Definition}
\theoremstyle{remark}
\newtheorem{rem}[thm]{Remark}
\numberwithin{equation}{section}
\newcommand{\set}[1]{\left\{#1\right\}}
\newcommand{\Real}{\mathbb R}

\newcommand{\func}[1]{\ensuremath{\mathrm{#1} \:} }

\newcommand{\dist}[0]{\mathrm{dist}}

\newcommand{\inj}[0]{\func{inj}}

\newcommand{\diam}[0]{\func{diam}}
\title{Compactness Properties of the space of genus-$g$ helicoids}
\author{Jacob Bernstein and Christine Breiner}
\address{Dept. of Math, Massachusetts Institute of
Technology, Cambridge, MA  02139, USA}
\email{jbern@math.mit.edu}

\address{Dept. of Math,
Johns Hopkins University, Baltimore, MD 21218, USA}
\email{cbreiner@math.jhu.edu}

\thanks{The first author was supported in part by NSF grant DMS 0606629}
\begin{document}
\begin{abstract}
In \cite{CM5}, Colding and Minicozzi describe a type of
compactness property possessed by sequences of embedded minimal
surfaces in $\Real^3$ with finite genus and with boundaries going to
$\infty$.  They show that any such sequence either contains a
sub-sequence with uniformly bounded curvature or the sub-sequence
has certain prescribed singular behavior.  In this paper, we
sharpen their description of the singular behavior when the
surfaces have connected boundary.  Using this, we deduce certain
additional compactness properties of the space of genus-$g$
helicoids.
\end{abstract}
\maketitle
\section{Introduction}
The goal of this paper is to better understand the finer geometric
structure of elements of $\mathcal{E}(1,g)$, the space of
genus-$g$ helicoids.   Here $\mathcal{E}(e,g,R)$ denotes the set
of smooth, connected, properly embedded minimal surfaces,
$\Sigma\subset\Real^3$, so that $\Sigma$ has genus $g$ and
$\partial \Sigma \subset
\partial B_R(0)$ is smooth, compact and has $e$ components. Every
element of $\mathcal{E}(1,g)=\mathcal{E}(1,g,\infty)$ is asymptotic
to a helicoid (see \cite{BB2}) and hence the terminology ``genus-$g$
helicoid'' is warranted. We approach this problem by showing certain
compactness properties for $\mathcal{E}(1,g)$, which ultimately
bound the geometry of elements of $\mathcal{E}(1,g)$.  In
\cite{BBg1Cpct}, it is shown that the space $\mathcal{E}(1,1)$,
modulo symmetries, is compact. When the genus is greater than one,
we cannot deduce such a nice result as we cannot rule out the
``loss'' of genus. Nevertheless, we will show that after a suitable
normalization, for any $g$, $\cup_{l=1}^g \mathcal{E}(1,l)$ is
compact. Indeed, we prove a slight generalization:
\begin{thm} \label{g2CpctnessThmInt}
Suppose $\Sigma_i \in \mathcal{E}(1,g,R_i)$ ($g\geq 1$) with $0\in
\Sigma_i$, $\inj_{\Sigma_i} (0)\leq \Delta$,
$\inf\set{\inj_{\Sigma_i}(q): q\in \mathcal{B}_\Delta(0)} \geq
\epsilon>0$, and $R_i/r_+(\Sigma_i)\to \infty$.  Then a sub-sequence
of the $\Sigma_i$ converges uniformly in $C^\infty$ on compact
subsets of $\Real^3$ with multiplicity one to a surface
$\Sigma_\infty\in \cup_{l=1}^g \mathcal{E}(1,l)$.
\end{thm}

We define $r_+(\Sigma)$ in Section \ref{TopDefSec}, noting now
only that it roughly measures the smallest extrinsic scale that
contains all of the genus. The normalization requires only that
the topology neither concentrates, nor disappears, near $0$. In
order to arrive at this result, we refine the powerful lamination
theory given by Colding and Minicozzi in \cite{CM5}. In its
simplest form -- i.e. Theorem 0.1 of \cite{CM4} -- the lamination
theorem states that a sequence of embedded minimal disks, with
boundaries going to $\infty$ and without uniformly bounded
curvature, must contain a sub-sequence converging to a foliation
of $\Real^3$ by parallel planes. Moreover, the convergence is in a
manner analogous to the homothetic blow-down of a helicoid.
Theorem 0.9 of \cite{CM5} generalizes this for sequences of
surfaces with more general topologies -- requiring only that the
surfaces are uniformly ``disk-like'' on small scales.
%Even more
%general results are obtained in \cite{CM5}, though we will not
%need them.
As Colding and Minicozzi's paper is somewhat involved,
we refer the reader to Appendix A of \cite{BBg1Cpct} which
provides a summary of the relevant definitions and results.  While
we make use of this lamination theory extensively, it is not
sufficiently precise for our purposes.  Thus, we prove the
following sharpening, when the boundaries are connected, which
describes in more detail the fate of the topology in the limit:

\begin{thm}\label{GenusLossThm}
Suppose $\Sigma_i \in \mathcal{E}(1,g,R_i)$ ($g\geq 1$), $R_i\to
\infty$, $r_+(\Sigma_i)=1$, the genus of each $\Sigma_i$ is
centered at $0$, and $\sup_{B_1(0)\cap \Sigma_i} |A|^2 \to
\infty$. Then, up to passing to a sub-sequence and rotating
$\Real^3$, the following holds:
\begin{enumerate}
 \item The $\Sigma_i$ converge to the lamination $\mathcal{L}=\set{x_3=t}_{t\in
\Real}$ with singular set $\mathcal{S}$ the $x_3$-axis in the sense of Theorem 0.9 of \cite{CM5}.
\item There is a number $2\leq l\leq g$ and a set of $l$ distinct points
$\mathcal{S}_{genus}=\set{p_1,\ldots, p_{l}}\subset \{(0,0,t)| -1
\leq t \leq 1\}$, with $p_1=(0,0,-1)$ and $p_{l}=(0,0,1)$, radii
$r_1, \ldots r_l>0$ and sequences $r_1^i, \ldots r_l^i\to 0$ so that
the genus of $B_{r_j}(p_j)\cap \Sigma_i$, $g_j$, is equal to the
genus of $B_{r_j^i}(p_j)\cap \Sigma_i$ and $g_1+\ldots +g_{l}=g$.
\item Each component of $B_{r_j}(p_j)\cap \Sigma_i$ and of
$B_{r_j^i}(p_j)\cap \Sigma_i$ has connected boundary.
\item If $B_{\rho}(y)\cap \cup_j B_{r_j^i}(p_j)=\emptyset$, then each component of $B_{\rho}(y)\cap \Sigma_i$ is a disk.
\end{enumerate}
\end{thm}
\begin{rem}
By the genus of $B_r(x)\cap \Sigma$ we mean the sum of the genus
of each component, where the genus of the component is the genus
of the compact surface obtained after gluing disks onto the
boundary.
\end{rem}

The points of $\mathcal{S}_{genus}$ are precisely where (all) the
topology of the sequence concentrates. Importantly, by looking near
points of $\mathcal{S}_{genus}$ and rescaling appropriately, we
construct a new sequence that either continues to satisfy the
hypotheses of Theorem \ref{GenusLossThm} or has uniformly bounded
curvature. This dichotomy will be fundamental in both the proof of
Theorem \ref{GenusLossThm}, which requires an induction on the
genus, and in its applications.
 Theorem \ref{GenusLossThm} is of independent interest as
it imposes some geometric rigidity for $\Sigma \in \mathcal{E}(1,g)$
when $g \geq 2$. Indeed, Theorem \ref{GenusLossThm} quantifies, in a certain sense, the way
$\mathcal{E}(1,g)$ could fail to be compact.

The bulk of this paper is the proof of Theorem \ref{GenusLossThm},
which is contained in Section \ref{GenusLossSec}. Unsurprisingly, we
rely heavily on Colding and Minicozzi's fundamental study of the
structure of embedded minimal surfaces in $\Real^3$. Indeed, a
weaker form of Theorem \ref{GenusLossThm} -- which allows for the
possibility that some topology does not ``collapse''-- is an
immediate consequence of their lamination theory of \cite{CM5}. This
is Proposition \ref{WeakGenusLossThm} below, which will be a step in
the proof.  In order to refine things, we make use of two other
important consequences of their work: the one-sided curvature
estimates of \cite{CM4} and the chord-arc bounds for minimal disks
of \cite{CY}. The techniques in the proof are very similar to those
used in \cite{BBg1Cpct}, though here the arguments are more
technical.  They are also similar to the arguments of
\cite{MPR3,MPR4,MRDuke}, though those papers have different goals.

Throughout we denote extrinsic balls in $\Real^3$, centered at $x$
and with radius $r$, by $B_r(x)$; intrinsic balls in a surface are
denoted by $\mathcal{B}_r(x)$. For a surface $\Sigma$, $|A|^2$
denotes the norm squared of the second fundamental form.  At various
points we will need to consider $\Sigma\cap B_r(x$) and when we do,
we always assume $\partial B_r(x)$ meets $\Sigma$ transversely as
this can always be achieved by arbitrarily small perturbations.

\section{Collapse of the Genus} \label{GenusLossSec}
In order to prove Theorem \ref{GenusLossThm} we will induct on the
genus.  When the genus is one, we can appeal to \cite{BBg1Cpct} to
show that the curvature is bounded uniformly and so Theorem
\ref{GenusLossThm} is vacuous.  The relevant result of
\cite{BBg1Cpct} is recorded as Theorem \ref{CpctnessCorg1} below.
When the genus is larger than one, the theorem will follow more or
less from the no-mixing theorem of \cite{CM5}, after one rules out
the possibility that there are handles in the sequence that do not
``collapse''.  The no-mixing theorem roughly states that, for points
in the singular set $\mathcal{S}$, the topology of the sequence must
behave uniformly in the same manner. Specifically, one cannot have a
sequence of minimal surfaces where near $x \in \mathcal{S}$ the
sequence is uniformly ``disk-like'' (i.e. $x \in
\mathcal{S}_{ulsc}$) whereas near $x \neq y \in \mathcal{S}$ it
looks uniformly ``neck-like'' (i.e. $y \in \mathcal{S}_{neck}$). If
there was a non-collapsed handle, then the nature of the singular
convergence would force it to lie nearer and nearer the singular
axis.  This contradicts certain chord-arc bounds for embedded
minimal surfaces and so cannot occur.  The arguments will be very
similar to those in Section 2.2 in \cite{BBg1Cpct}. Importantly, in
\cite{BBg1Cpct}, the sequence was simply connected on small uniform
scales which is not true in the present case. This introduces
technical difficulties.

%As a first step, we will apply Theorem
%\ref{WeakCompThmInt} to sequences of surfaces in
%$\mathcal{E}(1,g,R_i)$ which have a uniform control on the size of the topology.  For such %sequences, the limit must lie in $\mathcal{E}(1,g)$.  The main
%tool we will use to restrict the topology of the limit surfaces is
%the no-mixing theorem of \cite{CM5}. When $g=1$, such strong, two-sided control follows when
%either the intrinsic or extrinsic normalization is imposed.  Thus,
%for genus-one surfaces, Theorem \ref{GenusLossThm} is vacuous and so
%Theorems \ref{g2CpctnessThmInt} and \ref{g2CpctnessThmExt} are
%immediate.

\subsection{Topological definitions} \label{TopDefSec}We first
introduce a number of definitions and state
some simple propositions regarding the topological structure of
surfaces, $\Sigma \in \mathcal{E}(1,g,R)$.  These are all easy
consequences of the classification of surfaces. The first result
gives a basis for $H_1(\Sigma)$ in terms of embedded closed curves
with certain nice properties.
\begin{defn} Let $\Sigma \in \mathcal{E}(1,g,R)$. We call a collection of simple closed curves $\eta_1, \ldots, \eta_{2g}$ in $\Sigma$ that satisfies $\# \set{p| p \in \eta_i \cap \eta_j}=\delta_{i+g,j}$ a \emph{homology basis of $\Sigma$}.
\end{defn}
\begin{prop} \label{homolbasisprop} Any $\Sigma\in \mathcal{E}(1,g,R)$ contains a homology basis $\eta_1,\ldots, \eta_{2g}$. The homology classes $[\eta_i]$ generate $H_1(\Sigma)$. Furthermore, any closed curve $\eta\subset \Sigma \backslash \cup_i \eta_i$ is separating, that is $\Sigma \backslash \eta$ has at least two components.
\end{prop}
Another consequence is that we can decompose $\Sigma$ into once
punctured tori, which by abuse of terminology we refer to as
\emph{handles}.  To that end we introduce the following definition
and an immediate consequence:
\begin{defn}
We say a set $\set{\Sigma^1,\ldots, \Sigma^g}$ of pair-wise disjoint
surfaces is a \textit{handle decomposition} of $\Sigma\in
\mathcal{E}(1,g,R)$ if each $\Sigma^i\subset \Sigma$ is a compact genus
$1$ surface with connected boundary that contains closed curves
$\eta_{i},\eta_{i+g}$ so that $\eta_1,\ldots \eta_{2g}$ are a
homology basis of $\Sigma$.
\end{defn}
\begin{prop}
Let $\Sigma\in \mathcal{E}(1,g,R)$ and let $\eta_i$ be as above.
Then there are closed disjoint sub-surfaces of $\Sigma$,
$\Sigma^1,\ldots, \Sigma^g$, with connected boundary and genus one
such that $\Sigma^i$ contains $\eta_i, \eta_{i+g}$.  Moreover,
$\Sigma\backslash \cup_i \Sigma^i$ is a planar domain with $g+1$
boundary components.
\end{prop}

Continuing with our abuse of notation, we refer to $\Sigma^{k}$ as a
\emph{$k$-handle} if it is a compact genus $k$-surface with
connected boundary.  A \emph{generalized handle decomposition} of
$\Sigma \in \mathcal{E}(1,g,R)$ is a set $\set{\Sigma^{1,k_1},
\ldots , \Sigma^{l,k_l}}$ of pairwise disjoint subsets of $\Sigma$
so that each $\Sigma^{j,k_j}$ is a $k_j$-handle and
$k_1+\ldots+k_l=g$.

We now fix the language we will use to define the extrinsic scale(s)
of the genus:
\begin{defn}\label{outerScaleDef}
For $\Sigma\in \mathcal{E}(1,g,R)$ let
\begin{multline*}
    r_+(\Sigma)=\inf_{x\in B_R}\inf\left\{ r: B_r(x)\subset B_R \right. \mbox{and}  \left. B_r( x) \cap \Sigma \mbox{ has a component of genus $g$} \right\}.
\end{multline*}
We call $r_+(\Sigma)$ the \emph{outer extrinsic scale of the genus}
of $\Sigma$. Furthermore, suppose for all $\epsilon>0$, one of the
components of $B_{r_+(\Sigma)+\epsilon}(x)\cap \Sigma$ has genus
$g$; then we say the genus is \emph{centered at} $x$.
\end{defn}

The outer scale of the genus measures how spread out all the handles
are and the center of the genus should be thought of as a ``center
of mass" of the handles. We also need to measure the scale of
individual handles and to that end define:
\begin{defn} \label{InnerScaleLocDef}
For $\Sigma\in \mathcal{E}(1,g,R)$ and $x\in B_{R}$ let
\begin{equation*} r_-(\Sigma,x)=\sup \left\{ r: B_r(x)\subset
B_R(0) \mbox{ and } B_r( x) \cap \Sigma \mbox{ is genus zero} \right\}.
\end{equation*}
If the genus of $B_r(x)\cap \Sigma$ is zero whenever
$B_r(x)\subset B_R(0)$,  set $r_-(\Sigma,x)=\infty$. Define
$r_-(\Sigma)=\inf_{x\in B_R(0)} r_-(\Sigma,x)$.
\end{defn}

We recall a simple topological lemma that is a localization of
Proposition A.1 of \cite{BB2} and is proved using the maximum
principle in an identical manner.
\begin{lem} \label{LocTopProp}
Let $\Sigma\in \mathcal{E}(1,g,R)$ and suppose the genus is centered
at $x$.  If $\bar{B}_r(y)\cap \bar{B}_{r_+(\Sigma)}(x)=\emptyset$
and $B_r(y)\subset B_R(0)$, then each component of $B_r(y)\cap
\Sigma$ is a disk.  Moreover, if $\bar{B}_{r_+(\Sigma)}(x)\subset
B_r(y)\subset B_R(0)$, then one component of $B_r(y)\cap \Sigma$ has
genus $g$ and connected boundary and all other components are disks.
\end{lem}

\subsection{Uniform collapse}\label{unifcollapsesec}
In order to prove Theorem \ref{GenusLossThm} we will need to
distinguish between handles in the sequence that collapse and
those that do not. By ``collapsing'', we mean handles that are
eventually contained in arbitrarily small extrinsic balls. The
collapsed handles will be further divided into those that collapse
at a ``uniform'' rate and those that do not.  ``Uniform'' collapse
implies that the geometry becomes small in a manner that is
amenable to a blow-up analysis. To help motivate our definition of
uniform we recall Theorem 1.3 of \cite{BBg1Cpct}, which
essentially says that control on both scales of the
genus gives compactness.
\begin{thm} \label{CpctnessCorg1}
Suppose $\Sigma_i \in \mathcal{E}(1,g,R_i)$ are such that
$1=r_-(\Sigma_i)\geq \alpha r_+(\Sigma_i)$, the genus of each
$\Sigma_i$ is centered at $0$ and $R_i\to \infty$.  Then a
sub-sequence of the $\Sigma_i$ converges uniformly in $C^\infty$ on
compact subsets of $\Real^3$ and with multiplicity one to a surface
$\Sigma_\infty\in \mathcal{E}(1,g)$ and
$1=r_-(\Sigma_\infty)\geq \alpha r_+(\Sigma_\infty)$.
\end{thm}

We make the following technical definition that specifies when a $k$-handle in a sequence $\Sigma_i\in \mathcal{E}(1,g,R)$ collapses uniformly.  As a consequence we can study the handle uniformly on the scale of the collapse.  Notice that by the lamination theory of \cite{CM5} and Theorem \ref{CpctnessCorg1}, a curvature bound is equivalent to a lower bound on $r_-$.
\begin{defn}\label{unifcollapsedef}
Let $\Sigma_i\in \mathcal{E}(1,g,R)$ and let $\Sigma_i'\subset
\Sigma_i$ be a sequence of $k$-handles in $\Sigma_i$.  We say that
$\Sigma_i'$ \textit{collapse uniformly at rate $\lambda_i$ to a
point $p$} if there are sequences $0<r_i<R$ and $\lambda_i\to 0$
with $r_i/\lambda_i \to \infty$, and points $p_i\to p$ satisfying
$B_{r_i}(p_i)\subset B_R$,
 so that $\Sigma_i'-p_i \in
\mathcal{E}(1,k,2\lambda_i)$, $\Sigma_i'\subset \Sigma_i'{}'\subset
\Sigma_i$ with $\Sigma_i'{}' -p_i\in \mathcal{E}(1,k,r_i)$, the
genus of $\Sigma_i'{}'$ is centered at $p_i$ with
$r_+(\Sigma_i'{}')=\lambda_i$ and $\lambda_i \sup_{\Sigma_i'} |A|
\leq C<\infty$.
\end{defn}

As the name indicates, there is a uniformity to the geometry of such
a sequence of handles.  We make this more precise in the following
result.
\begin{lem}\label{UniformGeomCor} \label{UniformGeomLem}
Let $\Sigma_i \in \mathcal{E}(1,g,R)$ and suppose $\Sigma_i'\subset
\Sigma_i$ is a sequence of $k$-handles collapsing uniformly at rate
$\lambda_i$ to some point $p$.  Then $\limsup_{i\to\infty}
\lambda_i^{-1} \diam (\Sigma_i') <\infty$.  Further, there exists a
closed geodesic $\gamma_i\subset \Sigma_i$ homotopic to $\partial
\Sigma_i'$ so that $\limsup_{i\to \infty} \lambda_i^{-1}
\ell(\gamma_i)<\infty$ and $\limsup_{i\to \infty}
\lambda_i^{-1}\dist_{\Sigma_i} (\gamma_i, \Sigma_i')<\infty$.
\end{lem}
\begin{proof} We first prove the diameter bound by contradiction. To that end, assume there exists a
sub-sequence $\Sigma_i$ such that $\lim_{i \to \infty}\lambda_i^{-1}
\diam (\Sigma_i')=\infty$. Notice that $\tilde{\Sigma}_i =
\lambda_i^{-1}(\Sigma_i''-p_i)$ (where $p_i$ are as in the
definition) satisfy the conditions of Theorem \ref{CpctnessCorg1}
and so sub-sequentially converge in $C^\infty$ on compact subsets of
$\Real^3$ to a $\tilde{\Sigma}_{\infty}\in \mathcal{E}(1,k)$ that
satisfies $r_+(\tilde{\Sigma}_\infty)=1$. By Lemma \ref{LocTopProp},
there exists one component, $\tilde{\Sigma}_\infty^0$, of
$\tilde{\Sigma}_\infty \cap B_2$ with genus $k$. Set $D=\diam
\tilde{\Sigma}_\infty<\infty$.  For $i$ large, there exists some
component of each $\tilde{\Sigma}_i$, call it $\tilde{\Sigma}_i^0$,
so that $\tilde{\Sigma}_i^0 $ can be written as the graph of some
function $u_i$ over $\tilde{\Sigma}_\infty^0$ and $||u_i||_{C^2}\to
0$.  Thus, for sufficiently large $i$, $\diam(\tilde{\Sigma}_i ^0)
\leq 2D$, a contradiction.

The uniform diameter bound implies that there exists a curve
$\gamma_i'\subset \tilde{\Sigma}_i$ homotopic to $\partial
\tilde{\Sigma}_i'$ with
$\dist_{\tilde{\Sigma}_i}(\gamma_i',\tilde{\Sigma}_i')+\int_{\gamma_i'}
(1+|k_g|)\leq D'<\infty$. Lemma \ref{NoThinTubesLem} below allows us
to argue by direct methods that there exists a length minimizer,
$\gamma_i$, in the homotopy class of $\gamma_i'$ with
$\dist_{\tilde{\Sigma}_i}(\gamma_i,\gamma_i')< C(D')$.  This
proves the lemma.
\end{proof}

In the above proof we used Lemma 2.2 of \cite{BBg1Cpct}.  As we use
it extensively in this paper, we record it here:
\begin{lem} \label{NoThinTubesLem}
Let $\Gamma$ be a minimal surface with genus $g$ and with $\partial
\Gamma=\gamma_1\cup \gamma_2$ where the $\gamma_i$ are smooth and
satisfy $\int_{\gamma_i} 1+|k_g| \leq C_1$.
Then, there exists $C_2=C_2(g,C_1)$ so that $\dist_{\Gamma}
(\gamma_1,\gamma_2)\leq C_2$.
\end{lem}

Theorems \ref{GenusLossThm} and \ref{CpctnessCorg1} can now be
used together to show that once a sequence of surfaces has a
single collapsing handle (and thus unbounded curvature), then
there is a decomposition such that all handles in the sequence are
uniformly collapsing. This allows one to uniformly study the
geometry of the handles. As we will need this fact as a step in
the inductive proof of Theorem \ref{GenusLossThm}, we state and
prove it here.
\begin{prop} \label{GenusLossCor}
Suppose $\Sigma_i \in \mathcal{E}(1,g,R_i)$ ($g\geq 1$), $R_i\to
\infty$, $r_+(\Sigma_i)=1$, the genus of each $\Sigma_i$ is centered
at $0$, and $\sup_{B_1(0)\cap \Sigma_i} |A|^2 \to \infty$. Then, up
to passing to a sub-sequence and rotating $\Real^3$: There is a
$2\leq l \leq g$ and $l$ disjoint $k_j$-handles,
$\Gamma_i^{j,k_j}\subset \Sigma_i $, with $k_1+\ldots+k_l=g$ so that
the $\Gamma_i^{j,k_j}$ collapse uniformly at a rate $\lambda_i^j$ to
(not-necessarily distinct) points $p^j$ on the $x_3$-axis.
\end{prop}
\begin{proof}
We proceed by induction on $g$.  For $g=1$ as $r_+(\Sigma_i)=1$,
Theorem \ref{CpctnessCorg1} implies the statement is vacuous.  For
$g=2$, Theorem \ref{GenusLossThm} implies there are two handles
collapsing, one at $(0,0,1)$ and one at $(0,0,-1)$. Rescaling about
each point and applying Theorem \ref{CpctnessCorg1} shows they are
uniformly collapsing. We now fix $g>1$ and assume the conclusion is
true for $g'<g$. Theorem \ref{GenusLossThm} gives points (not
necessarily distinct) $p_1, \ldots, p_m$, radii $r_1, \ldots, r_m$
and subsets $\Gamma_i^j\subset \Sigma_i$, $1\leq j \leq m$ so that
$\Gamma_i^j-p_j\in \mathcal{E}(1,k_j,r_j)$ and $r_+(\Gamma_i^j)\to
0$ as $i \to \infty$. Notice that because $r_+(\Sigma_i)=1$ one must
have $k_j<g$. At each $p_j$, an appropriate translation and
rescaling gives a sequence that either satisfies the above
hypotheses or Theorem \ref{CpctnessCorg1}. Thus, either the
induction hypothesis or direct application of Theorem
\ref{CpctnessCorg1} implies that all the handles collapsing at $p_j$
are uniformly collapsing. As this is true for all $j$, we've proven
the corollary.
\end{proof}
\subsection{The proof of Theorem
\ref{GenusLossThm}}\label{genuslosssec} We first note that the
no-mixing theorem of \cite{CM5} implies a weaker version of Theorem
\ref{GenusLossThm}.  Compare with Theorem 0.1 of \cite{CM3}:
\begin{prop}\label{WeakGenusLossThm}
Suppose $\Sigma_i \in \mathcal{E}(1,g,R_i)$ ($g\geq 1$) and $R_i\to
\infty$, $r_+(\Sigma_i)=1$, the genus of each $\Sigma_i$ is centered
at $0$, and $\sup_{B_1(0)\cap \Sigma_i} |A|^2 \to \infty$. Then up
to passing to a sub-sequence and rotating $\Real^3$:
\begin{enumerate}
 \item  The $\Sigma_i$ converge to the lamination $\mathcal{L}=\set{x_3=t}_{t\in
\Real}$ with singular set $\mathcal{S}$ a single line parallel to the $x_3$-axis in the sense of Theorem 0.9 of \cite{CM5}.
\item There is a number $1\leq l\leq g$ and $l$ distinct points
$p_1,\ldots, p_{l}$ on $\mathcal{S}$,  radii $r_j>0$ and sequences
$r_j^i\to 0$ so that the genus of $B_{r_j}(p_j)\cap \Sigma_i$,
$g_j$, is equal to the genus of $B_{r_j^i}(p_j)\cap \Sigma_i$ and
$g_1+\ldots +g_{l}\leq g$.
\item Each component of
$B_{r_j}(p_j)\cap \Sigma_i$ and of $B_{r_j^i}(p_j)\cap \Sigma_i$ has
connected boundary.
\item There is a $\delta_0>0$ so for any $0<\delta<\delta_0$, if
$B_{\delta}(y)\subset B_{R_i}\backslash  \cup_{j=1}^l
B_{r_j^i}(p_j)$, then each component of $B_{\delta}(y)\cap \Sigma_i$
is a disk.
\end{enumerate}
\end{prop}

\begin{proof}
The no-mixing theorem of \cite{CM5} and the fact that
$r_+(\Sigma_i)=1$ imply that the sequence of $\Sigma_i$ is ULSC; for
the details we refer to Lemma 3.5 of \cite{BBg1Cpct}. Theorem 0.9 of
\cite{CM5} and Proposition 2.1 of \cite{BBg1Cpct} imply that up to
passing to a sub-sequence and rotating $\Real^3$, the $\Sigma_i$
converge to the claimed singular lamination --  see Remark A.4 of
\cite{BBg1Cpct}.

Lemma I.0.14 of \cite{CM3} implies that, up to passing to a
further sub-sequence, there are $l\leq g$ points $p_1,\ldots, p_l$
(fixed in $\Real^3$) so that $r_-(\Sigma_i, p_j)\to 0$ whereas for
any other point $x\in \Real^3$, $\liminf_{i\to \infty}
r_-(\Sigma_i,x)>0$. Notice that $l\geq 1$, as otherwise
$r_-(\Sigma_i)\geq \alpha>0$ for some $\alpha$ and so by Theorem
\ref{CpctnessCorg1} a sub-sequence of the $\Sigma_i$ would have
uniformly bounded curvature.  Thus, it remains to show that one
can find $r_j, r_j^i$ and $\delta_0$ with the claimed properties.

By the definition of ULSC sequences, for each $p_j$ there is a
radius $0<r_j<1$ and radii $r_j^i\to 0$ so that $B_{r_j}(p_j)\cap
{\Sigma}_i$ has the same genus, $g_i^j$, as $B_{r_j^i}(p_j) \cap
{\Sigma}_i$ and the boundary of each component of  $B_{r_j^i}(p_j)
\cap {\Sigma}_i$ is connected. We claim that there exists $r_j{}'
\leq r_j$ so that, after possibly passing to a sub-sequence, each
component of $B_{r_j{}'} (p_j)\cap {\Sigma}_i$ also has connected
boundary. Indeed, if this was not the case then one could find
$\tilde{r}^i_j\in (r_j^i,r_j)$ with $\tilde{r}^i_j\to 0$ and some
component of $B_{\tilde{r}^i_j}(p_j)\cap {\Sigma}_i$ having
disconnected boundary. But notice the genus of
$B_{\tilde{r}^i_j}(p_j)\cap {\Sigma}_i$ is equal to the genus of
$B_{r_j}(p_j)\cap {\Sigma}_i$. By definition, this would imply $p_j
\in \mathcal{S}_{neck}$, contradicting the no-mixing theorem. Now,
redefine $r_j$ so $r_j{}'=r_j$.

Now suppose there was no such $\delta_0$. Then there would exist a
sequence of points $y_{k}$, radii $\rho_k\to 0$, and $\Sigma_{i_k}$
so that $B_{\rho_k} (y_k)\subset B_{R_{i_k}}\backslash
\cup_{j=1}^{l} B_{r_j^{i_k}}(p_j)$, but one component of
$B_{\rho_k}(y_k)\cap \Sigma_{i_k}$ was not a disk.  By throwing out
a finite number of these we may assume $\rho_k\leq\frac{1}{2}
\min\set{1, r_1,\ldots, r_l}$. Notice that as each $\Sigma_i$ is
smooth and $i_k\to \infty$, by passing to a sub-sequence and
relabeling we may replace the $\Sigma_{i_k}$ by $\Sigma_k$. Lemma
\ref{LocTopProp} and the fact that $r_+(\Sigma_{k})=1$ imply $y_k\in
B_2$. Passing to a sub-sequence, $y_k\to y_\infty \in B_2$.
Similarly, because each component, $\Gamma$, of $B_{r_j}\cap
(\Sigma_{k}-p_j)$ is either a disk or an element of
$\mathcal{E}(1,g_{j}, r_j)$ with genus lying in $B_{r_j^k}$, Lemma
\ref{LocTopProp} and the hypothesis imply that $y_k \notin \cup_j
B_{r_j/2}(p_j)$ ($1\leq j\leq l$). As the genus only concentrates at
$p_1,\ldots, p_l$, $\Sigma_k\cap B_{\rho_k}(y_k)$ must have a
component with disconnected boundary, implying $y_\infty\in
\mathcal{S}_{neck}$. This contradicts the no-mixing theorem of
\cite{CM5}.
\end{proof}
\begin{cor}\label{CollapseHandleCor}
Suppose $\Sigma_i \in \mathcal{E}(1,g,R_i)$ ($g\geq 1$), $R_i\to
\infty$, $r_+(\Sigma_i)=1$, the genus of each $\Sigma_i$ is centered
at $0$ and $\sup_{B_1(0)\cap \Sigma_i} |A|^2 \to \infty$. Then, up
to passing to a sub-sequence, there exist $1\leq g' \leq g$,
$\delta_0>0$, and a handle decomposition $\Sigma_i^k \subset
\Sigma_i\cap B_2(0)$ with $1\leq k \leq g$ so that:
\begin{enumerate}
 \item For $1 \leq j \leq g'$, there
are points $p_j$ and radii $r_j^i\to 0$ so that $\Sigma_i^j\subset
B_{r_j^i}(p_j)$.
\item For $j>g'$, no non-contractible closed curve in $\Sigma_i^j$ lies in any
$B_{\delta_0}(y)$.
\end{enumerate}
\end{cor}
\begin{rem}
We refer to the $\Sigma_i^j$ for $1\leq j\leq g'$ as
\emph{collapsing handles} and to the $\Sigma_i^j$ for
$g'<k $ as \emph{non-collapsing handles}. Notice, points $p_j$
need not be distinct. Also, if $g'=g$ there are no non-collapsing handles.
\end{rem}

The main obstacle to proving Theorem \ref{GenusLossThm} is the
possible existence of non-collapsing handles in the sequence.  If
there is a non-collapsing handle, then the chord-arc bounds of
\cite{CY} give geodesic lassos (geodesics away from one point) with
uniform upper and lower bounds on their length. As in the proof of
Theorem 1.4 in \cite{BBg1Cpct} this will lead to a contradiction;
however there are several subtleties. One of these is the need to
find the correct closed geodesics. Because the injectivity radius
collapses at some points, one must be careful in the selection.
Ideally, one would choose a closed geodesic that was part of the
homology basis of a non-collapsing handle, and was a minimizer in
its homology class. However, one does not a priori have the
existence of such a sequence lying in a fixed extrinsic ball.
Nevertheless, if such a pathology occurs, then there is a different
sequence of closed geodesics with acceptable properties. This is the
content of the following lemma:

\begin{lem}\label{ClosedGeoLem} Let $\Sigma_i \in \mathcal{E}(1,g,R_i)$ ($g\geq 1$) be as in Corollary \ref{CollapseHandleCor} with collapsing
handles $\Sigma_i^1, \ldots, \Sigma_i^{g'}$ and non-collapsing
handles $\Sigma_i^{g'+1},\ldots, \Sigma_i^g$, $1\leq g'<g$. Suppose,
in addition, that every collapsing handle is a subset of some
uniformly collapsing $k_j$-handle $\Gamma_i^{j,k_j}$, $1 \leq j\leq
l$, which collapse to points $p_j$.  Then, up to passing to a
sub-sequence, there exist $0< r_0\leq R_0<\infty$ and closed
geodesics $\gamma_i\subset \Sigma_i\cap B_{R_0}$ with $\gamma_i
\not\subseteq \cup_j B_{r_0}(p_j)$ so that either:
\begin{enumerate}
 \item \label{Case1ClosedGeoLem}
For $1\leq j\leq l$, $\dist_{\Sigma_i} (\gamma_i, \Gamma_i^{j,k_j})
\to \infty$; or
\item  \label{Case2ClosedGeoLem} the $\gamma_i$ minimize in their homology
class, $[\gamma_i]$, a generator of $H_1(\Sigma_i^g)$.
\end{enumerate}
\end{lem}
\begin{proof}
By the chord-arc bounds of \cite{CY}, for every point $p\in
\Sigma_i^{g}$, $\inj_{\Sigma_i} (p) \leq 2\Delta_0$ and thus there
is a geodesic lasso, $\gamma'_{i,p}$, of length $4\Delta_0$ through
$p$ -- see Lemma 3.6 of \cite{BBg1Cpct}.  Using Lemma
\ref{NoThinTubesLem}, a direct argument gives a closed geodesic,
$\gamma_{i,p}$, in $\Sigma_i$ homotopic to $\gamma'_{i,p}$ and with
$\dist_{\Sigma_i} (\gamma_{i,p}', \gamma_{i,p})\leq C$ where
$C=C(\Delta_0)$. Thus, $\gamma_{i,p}\subset
\mathcal{B}^{\Sigma_i}_{C+8\Delta_0}(p)$.  As a consequence, if
there is a sequence of points $p_i\in \Sigma_i^g$ so that
$\dist_{\Sigma_i}(p_i, \Gamma_i^{j,k_j}) \geq C+8\Delta_0+d_i$ where
$d_i\to \infty$ then setting $R_0=2+4\Delta_0$ and
$\gamma_i=\gamma_{i,p_i}$, we see that Case
\eqref{Case1ClosedGeoLem} is satisfied.

On the other hand, if one cannot find such a sequence $p_i$, then
after passing to a sub-sequence, one has that $\limsup_{i\to \infty}
\dist_{\Sigma_i} (x, \cup_{j=1}^{l} \Gamma_i^{j,k_j}) =C'<\infty$
for all $x \in \Sigma_i^g$.  By Lemma \ref{UniformGeomLem}, there is
a value $D<\infty$ bounding the diameter of each $\Gamma_i^{j,k_j}$.
Thus, for $i$ sufficiently large, there are points $q_i^1, \ldots,
q_i^{l}$ so that $\Sigma_i^g\subset\cup_{j=1}^{ l}
\mathcal{B}^{\Sigma_i}_{D'} (q_i^j)$ where $D'=C'+D$.  As a
consequence, there is a closed, embedded, non-contractible curve,
$\gamma_i'$, in $\Sigma_i^g$, forming part of a homology basis of
$\Sigma_i^g$ and whose length is less than $2 l D'$. The length
bound and the fact that $r_+(\Sigma_i)=1$ implies that $\gamma_i'$
lies in $B_{1+2l D'}(0)$, as does any homologous curve of equal or
smaller length. We now minimize length in $[\gamma_i']$ and obtain
$\gamma_i$. With $R_0=1+2l D'$, these curves satisfy Case
\eqref{Case2ClosedGeoLem}.

Finally, we verify that $\gamma_i
\not\subseteq\cup_{j}B_{r_0}(p_j)$.  To that end, fix $r_0$ so that
$r_0\leq \frac{1}{2}\min\set{\delta_0,r_1,\ldots, r_l}$
%and $r_0\leq
%\frac{1}{2} \min \set{|p_j-p_k|: p_j\neq p_k}$ REDUNDANT
where
the $\delta_0$ and the $r_l$ are given by Theorem \ref{WeakGenusLossThm}.
 Thus, the balls $B_{r_0}(p_j)$ are pair-wise disjoint and so
it suffices to show $\gamma_i\not\subseteq B_{r_0}(p_j)$. Suppose
$\Omega_i$ was the component of $B_{r_0}(p_j)\cap \Sigma_i$
containing $\gamma_i$.  As $\Omega_i$ has non-positive curvature and
$\gamma_i$ is a closed geodesic, $\Omega_i$ cannot be a disk.
However, by the choice of $r_0$ it does have connected boundary, and
so we may take it to be a $k$-handle where $1\leq k <g$. We claim
that if the $\gamma_i$ satisfy either Case \eqref{Case1ClosedGeoLem}
or Case \eqref{Case2ClosedGeoLem}, then they must separate
$\Omega_i$ and thus $\Sigma_i$ as well. Indeed, it is clear in
either case that one can choose a homology basis of $\Omega_i$,
$\sigma_i^1, \ldots, \sigma_i^{2k}$, disjoint from $\gamma_i$.  In
Case \eqref{Case1ClosedGeoLem} this is because the $\gamma_i$ are
far from the topology of the $\Omega_i$ whereas in Case
\eqref{Case2ClosedGeoLem} this is a purely topological fact.  Thus,
$\gamma_i \subset \Omega_i \backslash \cup_j \sigma_i^j$ and so is
separating. For Case
\eqref{Case2ClosedGeoLem}, this contradicts $\gamma_i$ being part of a homology basis.

Thus, we deal only with Case \eqref{Case1ClosedGeoLem}. Replace
$\Omega_i$ by the component of $\Omega_i\backslash \gamma_i$
disjoint from the boundary.  As $r_0 < \delta_0$, all the handles of
$\Omega_i$ lie within uniformly collapsing $k$-handles. Thus, there
is at least one uniformly collapsing handle $\Gamma_i^{j,k_j}\subset
\Omega_i$. Using, $\Gamma_i^{j,k_j}$, let $\gamma_i'{}'$ be the
closed geodesic given by Lemma \ref{UniformGeomCor}.  Clearly, for
$i$ sufficiently large, $\gamma_i$ and $\gamma_i'{}'$ are disjoint.
Thus, the component of $\Omega_i\backslash \gamma_i'{}'$ that meets
$\gamma_i$ satisfies the hypotheses of Lemma \ref{NoThinTubesLem}.
This implies that there is an upper bound on the distance between
$\gamma_i$ and $\gamma_i'{}'$ and hence an upper bound on the
distance between $\gamma_i$ and $\Gamma_i^{j,k_j}$ which is a
contradiction.
\end{proof}

We now prove Theorem \ref{GenusLossThm}. We will proceed by
induction on the genus; in doing so we must treat the two cases of
Lemma \ref{ClosedGeoLem} separately.
\begin{proof}(Theorem \ref{GenusLossThm}): Note that if $g=1$ then the
theorem is vacuously true by Theorem \ref{CpctnessCorg1}.  If $g=2$
then by passing to a sub-sequence Proposition \ref{WeakGenusLossThm}
implies that either only one handle collapses at a point $p_1\in
\mathcal{S}$ or two different handles collapse at $(0,0,\pm 1)$. Any
other possibility is not compatible with $r_+(\Sigma_i)=1$. In the
latter case, the theorem follows easily and so we treat only the
former case. A rescaling and Theorem \ref{CpctnessCorg1} imply the
collapsing handle is, after passing to a sub-sequence, uniformly
collapsing.  Thus, Lemma \ref{ClosedGeoLem} gives a sequence of
closed geodesics, $\gamma_i$ in $\Sigma_i$ with uniform upper (and
lower) bounds on their length. Moreover, $\gamma_i\not\subseteq
B_{r_0}(p_1)$, where $r_0$ is given by the lemma.

Up to passing to a sub-sequence, Lemma 2.4 of \cite{BBg1Cpct}
guarantees that the $\gamma_i$ converge, in a Hausdorff sense, to a
bounded closed sub-interval of $\mathcal{S}$. By Proposition
\ref{WeakGenusLossThm}, as $\gamma_i \not\subseteq B_{r_0}(p_1)$,
this interval has positive length and at least one endpoint
$q_\infty$ of the interval is not in $B_{r_0/2}(p_{1})$. By a
reflection, we may assume it is the bottom endpoint. For
$\delta<\frac{1}{4}r_0 \leq \frac{1}{8}\delta_0$  ($\delta_0$ from
Proposition \ref{WeakGenusLossThm}) and $i$ sufficiently large, each
component of $B_\delta(q_\infty)\cap \Sigma_i$ is simply connected.
Thus, the argument of Lemma 2.5 of \cite{BBg1Cpct} can be applied
without change to give a contradiction.

We now assume that Theorem \ref{GenusLossThm} holds for all $g'<g$;
in particular, Proposition \ref{GenusLossCor} holds for all $g'<g$.
By Proposition \ref{WeakGenusLossThm} there are points $p_1,\ldots,
p_l$ at which the genus concentrates and a scale $\delta_0$ so that
the $\Sigma_i$ are, away from the $p_j$, uniformly disks on scales
smaller than $\delta_0$. Label the collapsing handles $\Sigma_i^1,
\ldots, \Sigma_i^{g'}$. We assume $1\leq g'<g$ as otherwise the
theorem follows easily.  We claim that each collapsing handle
can be chosen to belong to a uniformly collapsing $k_j$-handle $\Gamma_i^{j,k_j}$.
Indeed, Proposition \ref{WeakGenusLossThm}, implies that each
collapsing handle lies in a $\tilde{k}_j$-handle
$\tilde{\Gamma}_i^{j,\tilde{k}_j}$ that, after a translation, lies
in $\mathcal{E}(1,\tilde{k}_j,r)$ for some $r>0$ and which has
$r_+(\tilde{\Gamma}_i^{j,\tilde{k}_j})\to 0$. Thus, after rescaling, we see that it satisfies either the
hypotheses of Theorem \ref{GenusLossThm} or Theorem
\ref{CpctnessCorg1}.  In the latter case, the handle is itself
uniformly collapsing, while in the former, as $\tilde{k}_j<g$,
Proposition \ref{GenusLossCor} decomposes $\tilde{\Gamma}_i^{j,\tilde{k}_j}$ into
uniformly collapsing handles.

Appealing to Lemma \ref{ClosedGeoLem}, since some handle is not
collapsing, we are guaranteed the existence of a closed geodesic
$\gamma_i$ of uniformly bounded length. Again, Lemma 2.4 of
\cite{BBg1Cpct} implies that, up to passing to a sub-sequence, the
$\gamma_i$ converge in a Hausdorff sense to a bounded closed
sub-interval of $\mathcal{S}$ of positive length. Clearly, if one of
the endpoints of this interval was not in the set
$\set{p_1,\ldots,p_l}$, Proposition \ref{WeakGenusLossThm} gives a
uniform scale near the endpoint on which $\Sigma_i$ would be simply
connected; as above this would give a contradiction. Thus, up to
relabeling, we may take the endpoints of the interval of convergence
to be $p_1$ and $p_2$. We must now deal with the two cases of Lemma
\ref{ClosedGeoLem} separately.

\vskip .05in \noindent Case \eqref{Case1ClosedGeoLem}:

Suppose the $\gamma_i$ are intrinsically far from the collapsing
handles.  We claim that as long as $i$ is sufficiently large, every
point $q\in \gamma_i$ has $\inj_{\Sigma_i} (q)\geq
\frac{1}{4}\delta_0$.  Note that for $i$ sufficiently large we have
that $\dist_{\Sigma_i} (\gamma_i, \cup_j \Gamma_i^{j,k_j})\geq
2\delta_0$. Suppose there exists $q\in\gamma_i$ with
$\alpha_{i,q}=\inj_{\Sigma_i}(q)<\frac{1}{4} \delta_0$; then there
exists a geodesic lasso $\gamma_{i,q}$ through
$q$ with length $2\alpha_{i,q}$. One of the points where topology
collapses, $p_j$, must lie in $B_{\delta_0/2} (q)$ as otherwise for
$i$ very large the component of $B_{\delta_0/2} (q)\cap \Sigma_i$
containing $\gamma_{i,q}$ a disk. Thus, $\gamma_{i,q}\subset
\Omega_{i,q}$, a component of $B_{\delta_0} (p_{j})\cap \Sigma_i$.
By Corollary \ref{CollapseHandleCor}, $\gamma_{i,q}$ cannot be contained in a non-collapsing handle. Since $\gamma_{i,q}$ is
non-contractible and intrinsically near $q$, while $q$ is far from
the uniformly collapsing handles $\Gamma_i^{j,k_j}$, it must be separating.
This is impossible, to see this, replace $\Omega_{i,q}$ by the component of $\Omega_{i,q}\backslash
\gamma_{i,q}$ with connected boundary. Then $\Omega_{i,q}$ must contain some
uniformly collapsing $k$-handle, but if this occurs then Lemma
\ref{NoThinTubesLem} and Corollary \ref{UniformGeomCor} contradict
$\dist_{\Sigma_i}(q,\cup_j \Gamma_i^{j,k_j})\to \infty$, verifying the
claim.

As a consequence, by the weak chord-arc bounds of \cite{CY}, there
is a $\delta \in (0,\delta_0)$ so that, for $i$ sufficiently large,
for any $q\in \gamma_i$ the component of $B_{\delta}(q)\cap
\Sigma_i$ containing $q$ is a disk.  Now pick $q_i\in \gamma_i$ to
be the lowest point of $\gamma_i$ (i.e. $x_3(q_i)=\min_{q \in
\gamma_i} x_3(q)$).  Clearly, $q_i\to p_1$ the bottom point of the limit interval of the
$\gamma_i$.  As a
consequence, for any $\epsilon>0$ there is an $i_\epsilon$ large so
that for $i>i_\epsilon$, $B_{\delta/2}(q_\infty)\cap \Sigma_i$ has
at least two components, one non-simply connected and one containing $q_i$, that meet
$B_{\epsilon}(q_\infty)$.  By the maximum principle, and the above the component
containing $q_i$ is a disk. The one-sided curvature bounds of
\cite{CM4} imply that, as long as $\epsilon$ is sufficiently small,
there is a $c>1$ so that the component $\Sigma_i^0$ of
$B_{\delta/c}(q_\infty)\cap \Sigma_i$ containing $q_i$ has
$\sup_{\Sigma_i^0} |A|^2\leq C$.  Hence there is a uniform
$\rho<\delta$ and $i_0 \geq i_\epsilon$ so that, for $i \geq i_0$,
the component $\Sigma_i^G$ of $B_{\rho}(q_\infty)\cap \Sigma_i$
containing $q_i$ is the graph over $T_{q_i} \Sigma_i$ with small
gradient. By the lamination theorem $\Sigma_i^G$ must actually
converge to a subset of the plane $\set{x_3=x_3(q_\infty)}$.  This
contradicts $\gamma_i$ being a geodesic that converges to
$\mathcal{S}$.

\vskip .05in \noindent Case \eqref{Case2ClosedGeoLem}:

Suppose the $\gamma_i$ are part of a homology basis of the
$\Sigma_i$ and let $q_i \to q_\infty$ represent the lowest point of
the limit interval of the $\gamma_i$.   By relabeling we may take
$p_1=q_\infty$. Pick $r$ such that $r<\frac{1}{2}r_0\leq
\frac{1}{4}\min\set{r_1,\ldots, r_l}$.
% and, for $2\leq j\leq
%l$, $p_j\not \in B_{2r}(q_\infty)$. REDUNDANT
Here $r_0$ is given by Lemma
\ref{ClosedGeoLem} and the $r_j$ are given by Proposition
\ref{WeakGenusLossThm}. Let $p_+= \mathcal{S} \cap
\partial B_{r}(q_\infty)$ such that $x_3(p_+)>x_3(q_\infty)$.  Since
$\gamma_i$ is not contained in $B_{r}(q_\infty)$, let $\gamma_i^0$
be the connected component of $\gamma_i \cap B_{r}(q_\infty)$ that
contains $q_i$; for sufficiently large $i$, this intersection is
non-empty. Denote by $q_i^\pm$ the boundary points of $\gamma_i^0$.
Notice that for $i$ large, $\ell(\gamma_i^0) \geq r$.  Moreover, as
$\gamma_i$ is minimizing in its homology class, any curve $\sigma
\subset B_{3r/2}(q_\infty)$ with $\partial \sigma = \{q_i^\pm\}$
such that $\sigma \cup \gamma_i^0$ bounds a 2-cell has $\ell(\sigma)
\geq \ell(\gamma_i^0)$.

Arguing exactly as in the proof of Lemma 2.5 of \cite{BBg1Cpct}, the
points $q_i^+$ and $q_i^-$, connected by $\gamma_i^0$, can be
connected in $\Sigma_i\cap B_{r/2}(p_+)$ by a curve $\sigma_i$ with
$\ell(\sigma_i) \to 0$. This follows from Proposition
\ref{WeakGenusLossThm} since there exists $i'$ large such that, for
all $i \geq i'$, $B_{r/2}(p^+) \cap \cup_{j=1}^l B_{r_j^i}(p_j) =
\emptyset$. Thus, all components of $\Sigma_i \cap B_{r/2}(p_+)$ are
disks.

Now we show that $\sigma_i \cup \gamma_i^0$ is null-homologous, and
thus get a contradiction. First, by the choice of $r$, every
component of $B_{2r}(q_\infty) \cap \Sigma_i$ has connected
boundary. Since $\gamma_i^0, \sigma_i \subset B_{r/2}(p_+)\subset
B_{2r}(q_\infty)$, we let $\Gamma_i$ denote the connected component
of $\Sigma_i \cap B_{2r}(q_\infty)$ that contains $\gamma_i^0$ and
$\sigma_i$. As the genus of $\Gamma_i$ is contained within
$B_{r_i^1} (q_\infty)$ where $r_i^1\to 0$, and $\sigma_i \in
B_{r/2}(p_+)$, we can find a homology basis of $\Gamma_i$ disjoint
from $\sigma_i$. Such a homology basis can also be chosen disjoint
from $\gamma_i$, as $\gamma_i$ was initially part of a homology
basis of $\Sigma_i$ (and belonged to a non-collapsing handle). Thus,
$\gamma_i^0 \cup \sigma_i$ separates $\Gamma_i$ and therefore bounds
a 2-cell. That is, $\gamma_i^0$ is homologous to $\sigma_i$. Thus,
for $i$ sufficiently large, we get a contradiction.
\end{proof}
\section{Proof of Theorem \ref{g2CpctnessThmInt}} \label{G2Sec}
%[This section needs to be completely rewritten]

%\subsection{Intrinsic normalization}
Theorem \ref{GenusLossThm}, in particular the nature
in which handles collapse, immediately gives compactness results for
one-ended embedded minimal surfaces with uniform control on the
inner scale of the topology.  We describe this inner scale
intrinsically (one could also formulate such a control
extrinsically, but this would be more technical). For genus-one
surfaces, control on the inner scale of the genus automatically
implies control on the outer scale (as they are equal); moreover, an
easy argument relates this to intrinsic scales.
In particular,  Theorem \ref{g2CpctnessThmInt} follows immediately
from Theorem \ref{CpctnessCorg1}  for genus-one surfaces.  On the
other hand, when the genus is $\geq 2$, the possibility remains that
the outer scale is unbounded and so Theorem \ref{CpctnessCorg1}
cannot be immediately applied. However, in this case we can use
Theorem \ref{GenusLossThm} to argue inductively.

%\subsection{Intrinsic Normalization}
\begin{proof}
We proceed by induction on the genus.  If $g=1$ then let
$\tilde{\Sigma}_i =r(\Sigma_i)^{-1} (\Sigma_i-x_i)$, where the genus
of $\Sigma_i$ is centered at $x_i \in B_{r(\Sigma_i)}$.  Clearly,
$\tilde{\Sigma}_i$ satisfy the hypotheses of Theorem
\ref{CpctnessCorg1} and so a sub-sequence converges smoothly to some
$\tilde{\Sigma}_\infty \in \mathcal{E}(1,1)$. If $r(\Sigma_i)\to
\infty$, then for $-r(\Sigma_i)^{-1} x_i=y_i \in \tilde{\Sigma}_i$
one has $\inj_{\tilde{\Sigma}_i}(y_i)\to 0$, which contradicts the
convergence. If $r(\Sigma_i)\to 0$, then we claim that
%Proposition
%\ref{LocTopProp} and the Gauss-Bonnet theorem imply
there are points $p_i \in\tilde{\Sigma}_i$ with
$\inj_{\tilde{\Sigma}_i}(p_i)\geq \epsilon r(\Sigma_i)^{-1}$ and
$|p_i|$ uniformly bounded. Let $\tilde{\Sigma}^0_i$ represent the
component of $B_1\cap \tilde{\Sigma}_i$ containing the genus. If
$\mathcal{B}_{\Delta r(\Sigma_i)^{-1}}^{\tilde{\Sigma}_i}(y_i) \cap
\tilde{\Sigma}_i^0 \neq \emptyset$ then the claim is immediate by
hypothesis.  If not, then $\tilde{\Sigma}^0_i$ is a subset of one of
the components of $\tilde{\Sigma}_i\backslash\mathcal{B}_{\Delta
r(\Sigma_i)^{-1}}^{\tilde{\Sigma}_i}(y_i)$. If no such points $p_i$
exist satisfying the uniform lower bound, then for every $R$ there
exists $i_R$ such that, for all $i \geq i_R$, we have $B_R\cap
\mathcal{B}_{\Delta
r(\Sigma_i)^{-1}}^{\tilde{\Sigma}_i}(y_i)=\emptyset$. By Lemma
\ref{LocTopProp}, the geodesic lasso originating at $y_i$ must
surround the component of $B_R \cap \tilde{\Sigma}_i$ containing
$\tilde{\Sigma}_i^0$. The Gauss-Bonnet theorem then uniformly bounds
the total curvature of this element (independent of $R$) --
contradicting the fact that elements of $\mathcal{E}(1,1)$ have
infinite total curvature. Clearly, one cannot have such points $p_i$
as $\tilde{\Sigma}_\infty$ is not a disk. Thus, $r(\Sigma_i)$ is
uniformly bounded away from $0$ and $\infty$. This proves the
theorem when $g=1$.  We now assume that the theorem holds for all
$1\leq g'<g$ and use this to deduce that it also holds for $g$.

We consider three cases: First,  $\infty>\lim_{i\to \infty}
r_+(\Sigma_i)\geq \lim_{i\to \infty} r_-(\Sigma_i)>0$; second,
$\lim_{i \to \infty} r_+(\Sigma_i)=\infty$; third, $\lim_{i\to
\infty} r_+(\Sigma_i)<\infty$ but $\lim_{i\to \infty}
r_-(\Sigma_i)=0$.  In the first the theorem is an immediate
consequence of Theorem \ref{CpctnessCorg1}. In the second case we
let $\tilde{\Sigma}_i=r_+(\Sigma_i)^{-1} \Sigma_i$.  In this case
one has $\inj_{\tilde{\Sigma}_i}(0)\to 0$.  Hence, the curvature is
blowing up and so we may apply Theorem \ref{GenusLossThm}.  Notice
that $0\in \mathcal{S}_{genus}$.   As a consequence, there is a
$\delta>0$ so that the component of $B_{\delta}(0)\cap
\tilde{\Sigma}_i$ containing $0$ lies in $\mathcal{E}(1,g_i,\delta)$
where $g_i<g$. Thus, by passing to a sub-sequence we have that the
component $\Sigma_i'$ of $B_{\delta r_+(\Sigma_i)}\cap \Sigma_i$
that contains $0$ is an element of $ \mathcal{E}(1,g',\delta
r_+(\Sigma_i))$ where $g'<g$.  Clearly, $\Sigma_i'$ satisfies the
inductive hypotheses and so contains a sub-sequence smoothly
converging with multiplicity one to $\Sigma_\infty'\in
\mathcal{E}(1,g'{}')$ with $g'{}'\leq g'$.  Finally, notice that
$\Sigma'_\infty$ is properly embedded and the $\Sigma'_i$ converge
to $ \Sigma_\infty'$ with multiplicity one. Moreover, there is no
complete properly embedded minimal surface in
$\mathbb{R}^3\backslash \Sigma_\infty$.  Thus, for any fixed $R>0$,
and for $i$ sufficiently large, depending on $R$, $\Sigma_i\cap B_R=
\Sigma_i'\cap B_R$, and so $\Sigma_i$ converges to $\Sigma_\infty'$,
which proves the theorem.

In the third case we note that the curvature must be blowing up, as
otherwise $r_-(\Sigma_i)$ would be uniformly bounded below, and so
Theorem \ref{GenusLossThm} can be applied to the $\Sigma_i$. Indeed,
Proposition \ref{GenusLossCor} gives uniformly collapsing
$k_j$-handles $\Gamma_i^{j,k_j}$, collapsing at rate $\lambda_i^j$,
with $k_1+\ldots +k_l=g$. Arguing as above, there must be points
$p_i\in \mathcal{B}^{\Sigma_i}_\Delta(0)$ with
$\inj_{\Sigma_i}(p_i)\geq \epsilon$ but $\dist_{\Sigma_i}(p_i,
\Gamma_i^{j,k_j})\leq C\lambda_i^j$ for some $j$.  As before, by a
rescaling argument this gives an immediate contradiction.
\end{proof}

\bibliographystyle{amsplain}
\bibliography{thesisbib}

\end{document}